\documentclass[11pt, oneside]{amsart}

\usepackage{amsmath,ifthen, amsfonts, amssymb,
srcltx, 
   amsopn,tikz, tkz-euclide, url}
\usetikzlibrary{decorations.markings,arrows, cd}

 \usepackage{dutchcal}

\usepackage{graphicx}

\newcommand{\showcomments}{yes}
\renewcommand{\showcomments}{no}

\newsavebox{\commentbox}
\newenvironment{com}%
{\ifthenelse{\equal{\showcomments}{yes}}%
{\footnotemark
        \begin{lrbox}{\commentbox}
        \begin{minipage}[t]{1.25in}\raggedright\sffamily\tiny
        \footnotemark[\arabic{footnote}]}
{\begin{lrbox}{\commentbox}}}%
{\ifthenelse{\equal{\showcomments}{yes}}%
{\end{minipage}\end{lrbox}\marginpar{\usebox{\commentbox}}}
{\end{lrbox}}}

\usepackage{amsmath,amsfonts, amssymb, graphicx, tikz, bbding,ifthen,srcltx,amsopn, url}
\usepackage{bm}
\usetikzlibrary{decorations.markings,arrows, intersections}
\usetikzlibrary{decorations.pathreplacing}

\newcommand{\Z}{\mathbb Z}

\newcommand{\E}{\mathbb E}
\newcommand{\h}{\mathbcal h}

\DeclareMathOperator{\Hull}{Hull}

\DeclareMathOperator{\Min}{Min}
\DeclareMathOperator{\sk}{sk}
\DeclareMathOperator{\canc}{canc}
\newcommand{\dist}{\textup{\textsf{d}}}

\theoremstyle{definition}
\newtheorem{thm}{Theorem}[section]
\newtheorem{lem}[thm]{Lemma}
\newtheorem{prop}[thm]{Proposition}

\newtheorem{remark}[thm]{Remark}
\newtheorem{defn}[thm]{Definition}

\author{Kasia Jankiewicz}
\address{Department of Mathematics, University of Chicago, Chicago, Illinois, 60637}
\email{kasia@math.uchicago.edu}
\title{Lower bounds on cubical dimension of $C'(1/6)$ groups}

\begin{document}
\begin{com}
{\bf \normalsize COMMENTS\\}
ARE\\
SHOWING!\\
\end{com}
\begin{abstract}For each $n$ we construct examples of finitely presented $C'(1/6)$ small cancellation groups that do not act properly on any $n$-dimensional CAT(0) cube complex.
\end{abstract}

\maketitle

\section{Introduction}
Groups that satisfy the $C'(1/6)$ small cancellation condition were shown 
to act properly and cocompactly on CAT(0) cube complexes in \cite{WiseSmallCanCube04}. 
In this note we are interested in the minimal dimension of a CAT(0) cube complex 
that such groups act properly  on. 
\begin{defn} The \emph{cubical dimension} of $G$ is the infimum of the values $n$ 
such that $G$ acts properly on an $n$-dimensional CAT(0) cube complex. 
\end{defn}

Wise's complex is obtained from Sageev's construction~\cite{Sageev95} 
with walls joining the opposite sides in each relator 
(after subdividing each edge into two if necessary). 
However, its dimension is not in general optimal. 
For example, the dimension of the CAT(0) cube complex 
associated to the usual presentation for the fundamental group of the surface of genus $g\geq 2$ is $g$, 
while its cubical dimension equals $2$ 
as it acts on the hyperbolic plane with a CAT(0) square complex structure.

We prove the following:
\begin{thm}\label{thm:main}
For each $n\geq 1$ and each $p\geq 6$ there exists a finitely presented $C'(1/p)$ small cancellation group $G$ such that the cubical dimension of $G$ is greater than $n$.
\end{thm}

For $n=1$, the stronger form of Theorem~\ref{thm:main} was proved by Pride in \cite{Pride83}. 
He gives an explicit example of an infinite $C'(1/6)$ group with property $\text{FA}$. 
Pride's construction has been revisited in \cite{JankiewiczWise17}. 
We observe that the case $n=2$ can be deduced from the work of Kar and Sageev 
who study uniform exponential growth of groups acting freely on CAT(0) square complexes \cite{KarSageev16}. See Remark~\ref{rem:kar sageev}. 

The groups in our construction can be chosen to be \emph{uniformly} $C'(1/p)$, 
i.e.\ where the length of each piece is  less than $1/p$ of the minimum of the lengths of relators. 
The presentation complex of a uniform $C'(1/p)$ presentation can be ``folded'' to a complex that admits a CAT(-1) (so also CAT(0)) metric \cite{Gromov2004, Brown16} (see also \cite{Martin17} for the CAT(0) metric). In particular, if $n\geq 2$ the groups we construct have finite cubical dimension strictly greater than their CAT(0) and geometric dimensions which both equal $2$. Let us remind that the CAT(0) dimension of a group $G$ is the infimum dimension of a complete CAT(0) space with a proper $G$-action by semi-simple isometries. The groups we construct join the list of examples of groups with ``dimension gaps''. Brady-Crisp showed that certain Artin groups of geometric dimension $2$ have CAT(0) dimension~$3$ \cite{BradyCrisp02}. We note that by \cite{HuangJankiewiczPrzytycki16} none of these groups act properly and cocompactly on CAT(0) cube complexes, but it is unknown whether they admit proper non-cocompact actions on CAT(0) cube complexes. Bridson constructed a finitely presented torsion-free group of geometric dimension $2$ and of CAT(0) and cubical dimension $3$ which has an index $2$ subgroup of cubical dimension $2$ \cite{Bridson01}. Crisp gave examples of Bestvina-Brady kernels with CAT(0) dimension $3$ and geometric dimensions $2$ \cite{Crisp02}. The CAT(0) dimension of these groups does not drop as we pass to finite index subgroups.

The groups we construct can be also viewed as examples of hyperbolic groups ``non-acting'' on CAT(0) cube complexes (of given dimension). On the far side of the spectrum there are hyperbolic groups with property (T), 
such as uniform lattices in $Sp(n,1)$ and
random groups at density $d>1/3$ \cite{Zuk03}\cite{KotowskiKotowski13}.

This note is organized as follows. In Section~\ref{sec:ccc} we recall the classification of isometries of a CAT(0) cube complex with respect to hyperplanes. We refer to \cite{LS77} for the background on small cancellation theory. The very basic notions of small cancellation theory are also recalled in Section~3. In that section we also describe how to build a $C'(1/p)$ presentations where relators are positive products of given words. This technical result is applied in Section 4, which is the heart of the paper and contains the proof of Theorem~\ref{thm:main}. The argument heavily utilizes hyperplanes to create a dichotomy between free subsemigroups and subgroups having polynomial growth. The main ingredient of the proof of Theorem~\ref{thm:main} is Lemma~\ref{lem:main} which states that for any two hyperbolic isometries $a,b$ of an $n$-dimensional CAT(0) cube complex one of the following holds: $\langle a^N,b^N\rangle$ is virtually abelian for some $N=N(n)$, or there is a hyperplane stabilized by certain conjugates of some powers of $a$ or $b$, or there is a pair of words in $a,b$ of uniformly bounded length that generates a free semigroup.

\subsection*{Acknowledgements}I would like to thank my PhD advisors Piotr Przytycki and Daniel Wise. I would also like to thank Carolyn Abbott, Yen Duong, Teddy Einstein, Justin Lanier, Thomas Ng and Radhika Gupta for helpful discussion on \cite{KarSageev16}. Finally, I am grateful to the referees for all their comments and suggestions. The author was partially supported by (Polish) Narodowe Centrum Nauki, grant no.\ UMO-2015/18/M/ST1/00050.

\section{Isometries and hyperplanes in CAT(0) cube complexes}\label{sec:ccc}
In this section we recall relevant facts about isometries of CAT(0) cube complexes and collect some lemmas that will be used in the proof of Theorem~\ref{thm:main}. For general background on CAT(0) cube complexes and groups acting on them we refer the reader to \cite{Sageev14}.

Throughout the paper $X$ will be a finite dimensional CAT(0) cube complex. 
The set of all hyperplanes of $X$ is denoted by $\mathcal H (X)$.
We use letters $h,h^*$ to denote the halfspaces of a hyperplane $\h$, and $N(\h)$ to denote the closed carrier of $\h$, i.e.\ the convex subcomplex of $X$ that is the union of all the cubes intersecting $\h$. We say that a hyperplane $\h$ \emph{separates} subsets $A,B\subset X$, if $A\subset h$ and $B\subset h^*$.
The metric $\dist$ is the $\ell_1$-metric on $X$. All the paths we consider are combinatorial (i.e.\ concatenations of edges), all the geodesics are with respect to $\dist$, and all axes of hyperbolic isometries are combinatorial axes. The combinatorial translation length $\delta(x)$ of an isometry $x$ is defined as $\inf_{p\in X^0}\dist(p,xp)$. If $x$ acts without hyperplane inversions then the infimum is realized and $\delta(x^k) = k\delta(x)$ \cite{HaglundSemiSimple} (see also~\cite{Woodhouse16}). In particular, $x$ has an axis and any axis of $x$ is also an axis of $x^k$. The \emph{combinatorial minset} of $x$ is $$ \Min^0(x) = \{p\in X^0 : \dist(p, xp) = \delta(x)\}$$ 
where $X^0$ is the $0$-skeleton of $X$. Every $0$-cube $p$ of $\Min^0(x)$ lies on an axis of $x$ of the form $\bigcup_{i\in \mathbb Z}[x^{i}p,x^{i+1}p]$ where $[x^{i}p, x^{i+1}p]$ is any geodesic joining $x^ip$ and $x^{i+1}p$.
Let $n=\dim X$. Let $x$ be a hyperbolic isometry of $X$ and let $\h$ be a hyperplane. We recall the classification of isometries of a CAT(0) cube complex. More details can be found in \cite[Sec 2.4 and 4.2]{CapraceSageev2011}. 
\begin{itemize} 
\item $x$ \emph{skewers} $\h$ if $x^kh\subsetneq h$ for one of the halfspaces $h$ of $\h$ and some $k>0$. Equivalently, if some (equivalently, any) axis of $x$ intersects $\h$ exactly once.
\item $x$ is \emph{parallel} to $\h$ if some (equivalently, any) axis of $x$ is in a finite neighbourhood of $\h$. 
\item $x$ is \emph{peripheral} to $\h$ if $x$ does not skewer $\h$ and is not parallel to $\h$. Equivalently, $x^kh\subsetneq h^*$ for some $k>0$.
\end{itemize}

Note that the type of behaviour of $x$ with respect to $\h$ is commensurability invariant, i.e.\ $x^i$ has the same type as $x$ with respect to $\h$. The set of all hyperplanes in $X$ skewered by $x$ is denoted by $\sk(x)$.
The constant $k$ in the above definitions can be chosen to be at most $n$. Indeed, the $n+1$ hyperplanes $\{\h,x\h,\dots, x^n\h\}$ cannot all intersect in $X$ since $\dim X = n$. In particular, if $\h\in \sk(x)$ then $x^{n!}h\subset h$  for one of the halfspaces $h\in\h$ since  for an appropriate $k<n$ we have $x^kh\subset h$ and so $x^{n!}h\subset x^{(\frac{n!}{k}-1)k}h\subset \dots \subset x^kh\subset h$. Similarly, we have the following: 
\begin{lem}\label{lem:ramsey} There exists a constant $K_3 = K_3(n)$ such that for each hyperplane $\h$ in $X$ and an isometry $x$ there exist $k<k'\leq K_3$ such that the hyperplanes $\{\h, x^k\h, x^{k'}\h\}$ pairwise are disjoint or equal.
\end{lem}
The constant in the lemma is the Ramsey number $R(n+1, 3)$. Recall that the Ramsey number  $R(k,l)$ is the least integer such that the complete graph $R(k,l)$ with all edges colored blue or red either contains a blue $k$-clique, or a red $l$-clique. For the existence of Ramsey number for any integers $k,n$ see e.g.\ \cite{GrahamRothschildSpencer80}.

\begin{proof} Consider the graph $\Gamma$ whose vertices correspond to integers, and two integers $r, q$ are joined by an edge if and only if $x^r\h$ and $x^q\h$ are distinct and intersect. Cliques in $\Gamma$ correspond to collections of distinct pairwise intersecting hyperplanes. Let $K_3$ be the Ramsey constant for numbers $(n+1)$ and $3$. Since $X$ is $n$-dimensional,  there are no $(n+1)$-cliques in $\Gamma$. The induced subgraph of $\Gamma$ on vertices $[0,K_3-1]$ must contain a $3$-anticlique. This corresponds to a triple of hyperplanes $\{x^p\h,x^q\h, x^r\h\}$ where $p<q<r$ that pairwise are disjoint or equal. Hence the hyperplanes $\{\h, x^{q-p}\h, x^{r-p}\h\}$ are pairwise disjoint or equal. \end{proof}

In the above Lemma the hyperplanes $\h, x^k\h, x^{k'}\h$ are pairwise disjoint, or $x^{K_3!}$ stabilizes $\h$ (and the two cases are not mutually exclusive).

\begin{lem}\cite[Lem 12]{KarSageev16}\label{lem:ping pong}
Suppose $x$ and $y$ are hyperbolic isometries of $X$ and there exists a hyperplane $\h = (h, h^*)$ such that $xh\subset h$, $yh\subset h$ and $xh\subset yh^*$.  Then $x,y$ freely generate a free semigroup. See Figure~\ref{fig:ping-pong triple}.
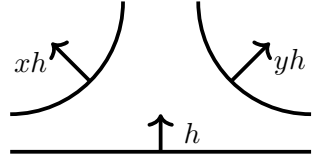
\begin{figure}\label{fig:ping-pong triple}
\begin{tikzpicture}
\draw[line width=0.5mm] (-2,0) to (2,0);
\draw[->, line width=0.5mm] (0,0) -- (0,0.5) node [midway, label=right:$h$] {};
\draw[line width=0.5mm] (-2,0.5) to[out=0, in=270] (-0.5,2);
\draw[line width=0.5mm] (2,0.5) to[out=180, in=270] (0.5,2);
\draw[->, line width=0.5mm] (-0.95,0.95) -- (-1.45, 1.45) node [midway, label=left:$xh$] {};
\draw[->, line width=0.5mm] (0.95,0.95) -- (1.45, 1.45) node [midway, label=right:$yh$] {};
\end{tikzpicture}
\caption{A ping-pong triple.}
\end{figure}
\end{lem}
 The triple $\{h, xh, yh\}$ as in Lemma~\ref{lem:ping pong} is called a \emph{ping-pong triple}. The following Lemma is a higher dimensional version of the All-Or-Nothing Lemma \cite[Lem 13]{KarSageev16}. Our proof is based on the proof of Kar--Sageev but it differs slightly.
 
\begin{lem}\label{lem:all or nothing}
Let $x$ and $y$ be hyperbolic isometries and let 
$\h\in\sk(x)$. Then one of the following holds
\begin{itemize}
\item $y$ skewers all $x^{in!}\h$ for $i\in\Z$, or 
\item $y$ skewers none of $x^{in!}\h$ for $i\in\Z$, or 
\item one of the following pairs of words freely generate a free semigroup for some $1\leq k\leq n$: 
\begin{align}\tag{$\star$}
\begin{split}
(x^{n!}, y^{kn!}x^{n!}),\\
(x^{n!}, y^{-kn!}x^{n!}),\\
(x^{-n!}, y^{kn!}x^{-n!}),\\
(x^{-n!}, y^{-kn!}x^{-n!}).\\
\end{split}
\end{align}
\end{itemize}
\end{lem}
\begin{proof}
Let $h$ be the halfspace of $\h$ such that $x^{n!}h\subsetneq h$. Suppose that $y$ skewers some hyperplane in $\{x^{in!}\h\}_{i\in\Z}$ but not all of them. Without loss of generality we can assume that $y$ skewers exactly one of $\h,x^{n!}\h$. First suppose $y$ skewers $\h$ but not $x^{n!}\h$  i.e.\ the axis $\gamma_y\subset x^{n!}h^*$. Since $\gamma_y$ goes arbitrarily deep in $h^*$ we have that $y$ is peripheral to $x^{n!}\h$. We either have $y^{n!}h\subset h$ or $y^{-n!}h\subset h$. 
Let $k$ be such that $y^{kn!}x^{n!}\h$ and $x^{n!}\h$ are disjoint. 
Either $y^{kn!}x^{n!}h\subset y^{kn!}h\subset h$ or $y^{-kn!}x^{n!}h\subset y^{-kn!}h\subset h$ and thus $\{h, x^{n!}h, y^{kn!}x^{n!}h\}$ or $\{h, x^{n!}h, y^{-kn!}x^{n!}h\}$ is a ping-pong triple.
Similarly, if $y$ skewers $x^{n!}\h$ but not $\h$, then one of $\{x^{n!}h^*, h^*, y^{kn!}h^*\}$ or $\{x^{n!}h^*, h^*, y^{-kn!}h^*\}$ is a ping-pong triple.
\end{proof}

The \emph{combinatorial convex hull} $\Hull(A)$ of a subset $A\subset X^0$ is the intersection of all convex subcomplexes containing $A$, i.e.\ $\Hull(A)$ is the maximal subcomplex contained in $\bigcap\{h:A\subset h\}$. 
Every convex subcomplex $Y$ in $X$ is a CAT(0) cube complex dual to the collection of hyperplanes that intersect $Y$. In particular, $\Hull(A)$ is a CAT(0) cube complex dual to the collection of hyperplanes $\mathcal H_A = \{\h\in\mathcal H(X): h\cap A \neq\emptyset\text{ and } h^*\cap A \neq\emptyset\}$.

 The first part of the lemma below, in the stronger version where $k=\dim\left( \Hull(\gamma_x)\right)$ but under the assumption that $x\h$ and $\h$ are disjoint, can be found in \cite[Proposition 5.4]{FernosForesterTao19}. The second part of the below lemma also follows from \cite[Theorem 2.1]{HaettelArtin}. Related results can be also found in \cite{Genevois19}.
\begin{lem}\label{lem:convex hull}
\item
\begin{enumerate}
\item The combinatorial convex hull $\Hull(\gamma_x)$ of an axis $\gamma_x$ of $x$ isometrically embeds in $\E^k$ for some $k\geq1$.
\item The $0$-skeleton of $\Hull(\Min^0(x))$ is contained in $\Min^0(x^{n!})$.
\end{enumerate}
\end{lem}
\begin{proof}
Let $p$ be some $0$-cube of $\gamma_x$. Let $\h_1,\dots, \h_k$ denote all the hyperplanes separating $p$ and $x^{n!}p$ (in particular, $k=n!\delta(x)$). Since $x^{n!}h_i\subset h_i$ for all $i$ and appropriate choice of halfspace $h_i$ of $\h_i$, the partition of the set of all hyperplanes skewered by $x$ into $\{x^{in!}\h_1\}_{i\in\Z}, \dots ,\{x^{in!}\h_k\}_{i\in\Z}$ gives an isometric embedding of $\Hull(\gamma_x)$ into a product of $k$ trees by \cite{ChepoiHagen13} where each tree is the cube complex dual to the collection of pairwise disjoint hyperplanes $\{x^{in!}\h_j\}_{i\in\Z}$. Since all these hyperplanes are intersected by a single bi-infinite geodesic (an axis of $x$), all the trees are in fact lines, i.e.\ $\Hull(\gamma_x)$ isometrically embeds in $\E^k$ with the standard cubical structure. The action of $x^{n!}$ extends to the action to $\E^k$ as a translation by the vector $[1,\dots, 1]$. Thus every $0$-cube of the combinatorial convex hull $\Hull(\gamma_x)$ is translated by $k=n!\delta(x) = \delta(x^{n!})$ and therefore the $0$-skeleton of $\Hull(\gamma_x)$ is contained in $\Min^0(x^{n!})$.

The subcomplex $\Hull(\Min^0(x))$ is dual to $\mathcal H_x =\{\h
: \Min^0(x) \cap h\neq\emptyset \text{ and } \Min^0(x) \cap h^*\neq\emptyset\}$. If $\h\notin\mathcal H_x - \sk(x)$, and $p,p'\in \Min^0(x)$ are separated by $\h$, i.e.\ $p\in h$ and $p'\in h^*$, then $x$ is parallel to $\h$. Indeed, $x^ip\in h$ and $x^ip'\in h^*$ for all $i$ and since $\dist(x^ip, x^ip') = \dist(p,p')$ the axis $\gamma_x$ through $p$ is contained in $N_d(\h)$ where $d\leq \dist(p,p')$. Thus the set $\mathcal H_x$ consists of hyperplanes skewered by $x$ or parallel to $x$. It follows that $\Hull(\Min^0(x))$ decomposes as a product $Y\times Y^{\perp}$ where $Y$ is dual to $\sk(x)$ and $Y^{\perp}$ is dual to the set of all the hyperplanes of $\mathcal H_x$ that are parallel to $x$ (see also \cite{KarSageev16}). For each $p\in Y^{\perp}$ the complex $Y\times\{p\}$ is the combinatorial convex hull of an axis of $x$. It follows that $\Hull(\Min^0(x))$ is the union of the complexes of the form $\Hull(\gamma_x)$ and so the $0$-skeleton of $\Hull(\Min^0(x))$ is contained in $\Min^0(x^{n!})$.

\end{proof}

\begin{lem}\label{lem:euclidean}
Let $X$ be a CAT(0) cube complex that is a subcomplex of a CAT(0) cube complex that is quasi-isometric to $\mathbb E^k$.
Then any finitely generated group $G$ for which there is a bound on the size of its finite subgroups, that acts properly on $X$, 
is virtually abelian. 
\end{lem}

\begin{proof}
The growth of $X^0$ is a polynomial of degree at most $k$ and so is the growth of $G$. Hence $G$ is virtually nilpotent, and by \cite[Thm 7.16]{BridsonHaefliger} $G$ is virtually abelian.
\end{proof}

%

\section{Constructing small cancellation presentations}\label{sec:small cancellation}
The main goal of this section is the following. 
\begin{prop}\label{prop:small cancellation presentation} Let $p\geq 6$. Let $\mathcal U = \{(u_i, v_i)\}_{i=1}^m$ be a finite collection of pairs where for each $i$ the elements $u_i, v_i\in F(x,y)$ are not powers of the same element. There exists a $C'(1/p)$ small cancellation presentation 
$$\langle x,y \mid r_1,\dots ,r_m\rangle$$
where $r_i$ is a positive word in $u_i, v_i$ that is not a proper power for $i=1,\dots, m$.
\end{prop}

By $F(x,y)$ in the above Lemma and throughout the section we denote the free group on generators $x$ and $y$. The length of a word $u$ with respect to $x,y$ is denoted by $|u|$. A \emph{spelling} of a nontrivial element $u\in F(x,y)$ is a concatenation $u_1\cdots u_m = u$ where each \emph{syllable} $u_i$ is a nontrivial element of $F(x,y)$. The \emph{cancellation} in the spelling $uv$ is the value $\canc(u,v) = \frac1 2 (|u|+|v|-|uv|)$, i.e.\ the length of the common prefix of the reduced words representing $u^{-1}$ and $v$. A spelling is \emph{reduced} if $\canc(u_i, u_{i+1}) = 0$ for $i=1,\dots, m-1$; in other words $|u| = \sum_i |u_i|$. A spelling is \emph{cyclically reduced} if additionally $\canc(u_{m}, u_1) = 0$. For $u,v\in F(x,y)$ we say $u,v$ are \emph{virtually conjugate} and write $u\sim v$ if some powers of $u$ and $v$ are conjugate. We denote a free semigroup on $u,v$ by $\{u, v\}^+$. Let $u^*$ denote 
an element $u^k$ for some $k\geq 0$.

A \emph{piece} in $u$ and $v$ is a syllable $w$ such that $wu'$ and $wv'$ are reduced spellings of some cyclic permutation of $u$ and $v$ respectively. We emphasize that all the pieces considered throughout the section are words in $x,y$. A presentation $\langle x,y\mid  r_1,\dots ,r_m \rangle$ is \emph{$C'(1/p)$ small cancellation} if for every piece $w$ in $r_i$ and $r_j$ we have $|w|<\frac 1 p |r_i|$. For more background on small cancellation theory, see \cite{LS77}.

\begin{lem}\label{lem:C-to-1} Let $H$ be a finitely generated subgroup of $F_k$. There exists a constant $C=C(H<F_k)$ such that the map between the conjugacy classes of maximal $\Z$-subgroups induced by the the inclusion $H\hookrightarrow F_k$
\begin{com}the maximal $\Z$-subgroup in $H$ is contained not necessarily equal to a maximal $\Z$ subgroup in $F_k$.\end{com} is at most $C$-to-$1$.
\end{lem}

\begin{proof} Let $A\looparrowright B$ be an immersion of graphs where $B$ is a wedge of $k$ circles, where the induced map on the fundamental groups is the inclusion $H\hookrightarrow F_k$. 

For any graph $\Gamma=A,B$, the conjugacy class of a $\Z$-subgroup in $\pi_1 \Gamma$ can be represented by an immersion $L\looparrowright \Gamma$ of a line that factors as $L\looparrowright S\looparrowright \Gamma$ where $S$ is a circle, taken modulo the orientation. Thus different conjugacy classes of $\Z$-subgroups in $H$ that map into the same conjugacy class in $F_k$ are different lifts 
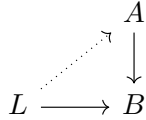
\[\begin{tikzcd}&A \arrow[d] \\L\arrow[r] \arrow[ru, dotted] & B\end{tikzcd}\] The number of such lifts is bounded by the number of vertices in $A$.
\end{proof}

\begin{lem}\label{lem:virtually conjugate} Let $u,v\in F(x,y)$ be such that $u$ and $v$ are not powers of the same element. There are infinitely many pairwise non virtually conjugate elements of the form $u^kv^k$.\end{lem}

\begin{proof} Two elements of $F(x,y)$ are virtually conjugate if and only if they have the following reduced spellings
\[
gw^ig^{-1}
\]
\[
h\bar w^jh^{-1}
\]
where $g,h$ are reduced words in $x,y$, $w$ is cyclically reduced and $\bar w$ is a cyclic permutation of $w$ \cite[Prop 2.17]{LS77}. In particular the elements of the set $\{x^ky^k:k\in\Z\}$ are not virtually conjugate, i.e.\ they are contained in distinct conjugacy classes of maximal $\Z$-subgroups. Since $u,v$ are not powers of the same element, the group $\langle u,v\rangle$ is a rank $2$ free group. By Lemma~\ref{lem:C-to-1} there exists a constant $C$ such that the map between the conjugacy classes of maximal $\Z$-subgroups induced by the inclusion $\langle u,v\rangle \hookrightarrow F(x,y)$ is at most $C$-to-$1$. The lemma follows.
\end{proof}

We say that elements $u,v\in F(x,y)$ are \emph{non-cancellable}, if 
\[\canc(u,v) < \frac 1 2 \min\{|u|, |v|\},\]
\[\canc(v,u) < \frac 1 2 \min\{|u|, |v|\}.\]
Equivalently,
for any $w_1,w_2\in \{u,v\}^+$ we have $\canc(w_1,w_2) < \frac 1 2 \min\{|u|, |v|\}$. In particular, $|w_1w_2|\geq \max\{|w_1|,|w_2|\}$. The equivalence is obvious in one of the direction. For the other direction, note that the cancellation between consecutive syllables of $w_1, w_2$ are separated by at least one letter. We remark that if $u,v$ are non-cancellable then so are any two elements in $\{u,v\}^+$.
\begin{lem}\label{lem:less than half cancellation}
Let $u, v\in F(x,y)$ not be powers of the same element. Then there exist elements $u',v'\in\{u,v\}^+$ that are non-cancellable and are not powers of the same element.
\end{lem}
\begin{proof}
If $\canc(u,v)> \frac 1 2 \min\{|u|, |v|\}$ replace the pair $(u,v)$ with $(u,uv)$ if $|u|\leq|v|$, and with $(v,uv)$ otherwise. If $\canc(v,u) >\frac 1 2 \min\{|u|, |v|\}$ replace the pair $(u,v)$ with $(u,vu)$ if $|u|\leq |v|$, and with $(v,vu)$ otherwise. Repeat these steps until $\canc(u,v),\canc(v,u)\leq\frac 1 2 \min\{|u|,|v|\}$. Since at each step the value $|u|+|v|$ strictly decreases, the procedure terminates in finitely many steps.
The fact that $u,v$ are not powers of the same element is equivalent to $\langle u,v \rangle \neq \Z$. Since $\langle u, uv\rangle = \langle u,v\rangle$ we conclude that $u, uv$ are not powers of the same element either. We argue the same way in all the above cases. 

Note that for any nontrivial element $w\in F(x,y)$ we have $\canc(w,w)<\frac12 |w|$, i.e.\ $|w^2|>|w|$. Let $u' = u^2$ and $v'=v^2$. Since $\canc(u,v),\canc(v,u)\leq \frac 12\min\{|u|, |v|\}$ and so $u'$ and $u$ have the same common prefix with $v^{-1}$ (and also with $(v')^{-1}$) we have $\canc(u',v') = \canc(u,v)\leq\frac 12 \min\{|u|,|v|\}) <\frac 12\min\{|u'|,|v'|\}$ as wanted. Similarly, $\canc(v',u') <\frac 12 \min\{|u'|,|v'|\}$. It follows that $\canc (w_1, w_2)<\frac 1 2 \min\{|u'|,|v'|\}$ for every $w_1,w_2\in\{u',v'\}^+$.
\end{proof}

\begin{lem}\label{lem:long overlap}
Let $s,t$ be two cyclically reduced elements in $F(x,y)$ such that $|s|\geq |t|>0$ such that $s^2$ is a prefix of $t^*$. Then $s,t$ are powers of the same element.
\end{lem}
\begin{proof} 
Suppose that $s$ and $t$ are not powers of the same element. In particular, $s$ is not a power of $t$, so there exists a nonempty prefix $w$ of $s$ that is both some prefix of $t$ and some suffix of $t$. See Figure~\ref{fig:conjugate elements}. 
\begin{figure}\begin{tikzpicture}
\draw (0,0) to (3,0) to (3,.5) to (0,.5) to (0, 0);
 \coordinate [label=$s$] (x) at (1.5,0);
\draw (3.1,0) to (6.1,0) to (6.1,.5) to (3.1,.5) to (3.1, 0);
 \coordinate [label=$s$] (x) at (4.6,0);
\draw (0,-0.2) to (2,-0.2) to (2,-.7) to (0,-.7) to (0, -.2);
 \coordinate [label=$t$] (x) at (1,-.75);
\draw (2.1,-.2) to (4.1,-.2) to (4.1,-.7) to (2.1,-.7) to (2.1, -.2);
 \coordinate [label=$t$] (x) at (3.1,-.75);  
 \draw[red, thick] (3.1,-.1) to (4.1,-.1);
  \draw[red, thick] (0,-.1) to (1,-.1);
 \end{tikzpicture}
\caption{Long overlap between $s^*$ and $t^*$. The red path is $w$.}\label{fig:conjugate elements}

\end{figure}
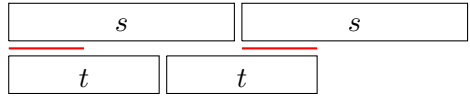
If $|w|\leq\frac1 2|t|$, then $t$ has a reduced spelling $wuw$ for some $u$, and $s$ has a reduced spelling $(wuw)^kwu$ for some $k\geq 1$. Then $s^2$ has a prefix $(wuw)^{k}wu\cdot wuww = (wuw)^{k+1}uww$ which thus must also be a prefix of $t^*$, and so it must coincide with $t^{k+1} = (wuw)^{k+2}$. In particular, $uw = wu$, which means that $w,u$ are powers of the same element. That is a contradiction. 

If $|w|>\frac 12 |t|$, then $t$ has reduced spellings $uw$ and $wu'$ for some $u,u'$ such that $|u|=|u'|<|w|$, and $s$ has a reduced spelling $(uw)^ku$ for some $k\geq 1$. The prefix $(uw)^kuuw$ of $s^2$ must coincide with the prefix $(uw)^{k+1}u$ of $t^{k+2}$. In particular $uw = wu$, which again is a contradiction.
\end{proof}

\begin{lem}\label{lem:small cancellation pair}
Let $u_i, v_i\in F(x,y)$ for $i=1,2$ where for each $i=1,2$ the elements $u_i, v_i$ are non-cancellable and are not powers of the same element.
Then for each $i=1,2$ there exist $s_i, t_i\in \{u_i, v_i\}^+$ 
such that
\begin{itemize}
\item $s_i,t_i$ are non-cancellable and are not virtually conjugate,
\item $\canc(s_i, t_i) = \canc(t_i, s_i) = \canc(s_i, s_i) = \canc(t_i, t_i)$, i.e.\ there exists $g_i$ such that $s_i = g_i\bar s_ig_i^{-1}$ and $t_i = g_i\bar t_ig_i^{-1}$ are reduced spellings where $\bar s_i$, $\bar t_i$ are cyclically reduced and have no cancellation,
\item for every piece $w$ between a word in $\{s_1, t_1\}^+$ and a word in $\{s_2,t_2\}^+$ we have $|w|<\min\{|s_i|, |t_i|\}$ for $i=1,2$.
\end{itemize}
\end{lem}

\begin{proof}
Since $u_i, v_i$ are non-cancellable,
 the consecutive cancellations between syllables in any word $r\in \{u_i, v_i\}^+$ are separated from each other. For $i=1,2$ set $s_i'=u_i^{n_1^i}v_i^{n_1^i}$ and $t_i'=u_i^{n_2^i}v_i^{n_2^i}$ where $n_1^1, n_1^2, n_2^1, n_2^2$ are chosen so that $s_1', t_1', s_2', t_2'$ are pairwise non virtually conjugate. This can be done by Lemma~\ref{lem:virtually conjugate}. Note that for $i=1,2$ we have $\canc(s_i', t_i') = \canc(t_i', s_i') = \canc(s_i', s_i
') = \canc(t_i', t_i') = \canc(v_i, u_i)$.
Any positive word $r(s_i',t_i')$ in $s_i', t_i'$ has the reduced spelling $g_ir(\bar s_i', \bar t_i')g_i^{-1}$.

Let $N=8\max\{|\bar s_1'|, |\bar t_1'|, |\bar s_2'|, |\bar t_2'|\}$ and set $s_i = (s_i')^N$ and $t_i = (t_i')^N$. Let $w$ be a piece between a word $r_1$ in $\{s_1,t_1\}^+$ and a word $r_2$ in $\{s_2, t_2\}^+$ and suppose that $|w|\geq N$. Note that the length of $w$ is short in comparison to the length of syllables of the form $(\bar s_i')^*, (\bar t_i')^*$, and that the initial or final subwords of the form $g_i, g_i^{-1}$ are shorter than $\frac 1 {16} N$. In particular, $w$ can overlap at most two syllables of the form $(\bar s_i')^*, (\bar t_i')^*$. Thus there exists a subword $w'$ of $w$ of length $\geq \frac 12 N$ that is a subword of $(\bar s_1')^*$ or of $(\bar t_1')^*$. For the same reason, there exists an even shorter subword $w''$ of $w'$ of length $\geq \frac 14 N$ that is also a subword of either $(\bar s_2')^*$ or of $(\bar t_2')^*$. Thus one of $(\bar s_1')^*, (\bar t_1')^*$, say $(\bar s_1')^*$ and one of $(\bar s_2')^*, (\bar t_2')^*$, say $(\bar t_2')^*$ have a common subword of length $\geq 2\max\{|\bar s_1'|, |\bar t_1'|, |\bar s_2'|, |\bar t_2'|\}$. In particular, some cyclic permutation of $(\bar s_1')^*$ is a subword of $(\bar t_2')^*$ (and vice-versa) and by Lemma~\ref{lem:long overlap} $\bar s_1', \bar t_2'$ are virtually conjugate. This is a contradiction. Thus $|w|<N$. We clearly also have $|s_i| = |(s_i')^N| \geq N$, and $|t_i| = |(t_i')^N| \geq N$ for $i=1,2$, and thus we get $|w|<\min\{|s_i|, |t_i|\}$.

\end{proof}

\begin{lem}\label{lem:pieces in single relator}
Let $s,t$ be cyclically reduced elements that are not proper powers in $F(x,y)$ such that $s,t$ are not virtually conjugate. Let $r = s^{\alpha_1}t^{\beta_1}\cdots s^{\alpha_{2p}}t^{\beta_{2p}}$ for some $p\geq 1$. If $\alpha_j, \beta_j$ are all different and greater than $2\max\{|s|,|t|\}+1$, then for every piece $w$ in $r$ we have $|w| \leq \left(\max\{\alpha_j\}+2\right)|s| + \left(\max\{\beta_j\}+2\right)|t|$.
\end{lem}
\begin{proof}

Consider two subwords of $r$: $\eta_0\eta_1\cdots \eta_k\eta_{k+1}$ and $\mu_0\mu_1\cdots\mu_{\ell}\mu_{\ell+1}$ where $\eta_i, \mu_j\in\{s,t\}$ such that $\eta_{1}\cdots \eta_{k}$ and $\mu_1\cdots\mu_{\ell}$  are maximal words in syllables $s,t$ entirely contained in $w$, i.e.\ each of these words is equal to $w$ after adding some prefix and suffix that are proper subwords of $s^{\pm}, t^{\pm}$. We say that two syllables $\eta_i$ and $\mu_j$ are \emph{aligned} if $\eta_i =\mu_j$ and they entirely overlap in $w$, i.e. they are the same subword of $w$.

Suppose two syllables $\eta_i,\mu_j $ overlap in $w$ and $\eta_i=\mu_j =s$. If they are not aligned, say a proper suffix of $\eta_i$ equals a proper prefix of $\mu_j$ then $\eta_{i+1}=t$ and $\mu_{j-1}=t$ (since $\mu_{j-1}=s$ or $\eta_{i+1}=s$ would imply that $s$ is equal to its conjugates which is not the case by Lemma~\ref{lem:long overlap} and the assumption that $s,t$ are not proper powers). See Figure~\ref{fig:piece in r}. 
\begin{figure}
\begin{tikzpicture}
\draw (-2.2,0) to (-1.2,0) to (-1.2,.5) to (-2.2,.5) to (-2.2, 0);
 \coordinate [label=$s$] (x) at (-1.7,0);
\draw (-1.1,0) to (-.1,0) to (-.1,.5) to (-1.1,.5) to (-1.1, 0);
 \coordinate [label=$s$] (x) at (-.6,0);
\draw (0,0) to (1,0) to (1,.5) to (0,.5) to (0, 0);
 \coordinate [label=$s$] (x) at (0.5,0);
\draw (1.1,0) to (2.3,0) to (2.3,.5) to (1.1,.5) to (1.1, 0);
 \coordinate [label=$t$] (x) at (1.7,0);
\draw (2.4,0) to (3.6,0) to (3.6,.5) to (2.4,.5) to (2.4, 0);
 \coordinate [label=$t$] (x) at (3,0);
 \draw (3.7,0) to (4.9,0) to (4.9,.5) to (3.7,.5) to (3.7, 0);
 \coordinate [label=$t$] (x) at (4.3,0); 
\draw (.5,-0.2) to (1.5,-0.2) to (1.5,-.7) to (.5,-.7) to (.5, -.2);
 \coordinate [label=$s$] (x) at (1,-.75);
 \draw (1.6,-0.2) to (2.6,-0.2) to (2.6,-.7) to (1.6,-.7) to (1.6, -.2);
 \coordinate [label=$s$] (x) at (2.1,-.75); 
 \draw (2.7,-0.2) to (3.7,-0.2) to (3.7,-.7) to (2.7,-.7) to (2.7, -.2);
 \coordinate [label=$s$] (x) at (3.2,-.75);
 \draw (-.8,-0.2) to (.4,-0.2) to (.4,-.7) to (-.8,-.7) to (-.8, -.2);
 \coordinate [label=$t$] (x) at (-.2,-.75); 
 \draw (-2.1,-0.2) to (-.9,-0.2) to (-.9,-.7) to (-2.1,-.7) to (-2.1, -.2);
 \coordinate [label=$t$] (x) at (-1.5,-.75);
 \draw (-3.4,-0.2) to (-2.2,-0.2) to (-2.2,-.7) to (-3.4,-.7) to (-3.4, -.2);
 \coordinate [label=$t$] (x) at (-2.8,-.75);
 \draw[red, thick] (-1.6,-.1) to (3.1,-.1);
 \end{tikzpicture}
 \caption{The red line is a maximal piece in $r$.}\label{fig:piece in r}
\end{figure}
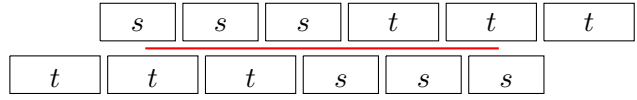
Since $s,t$ are not conjugate by Lemma~\ref{lem:long overlap} we get that $j\geq \ell-1$ and $i\leq 2$. That together with the fact that the prefixes and suffixes of $w$ that are excluded in $\eta_1\cdots \eta_k$ and $\mu_1\cdots\mu_{\ell}$ both have lengths less than $\max\{|s|,|t|\}$ implies that $|w|<6\max\{|s|,|t|\}<6(|s|+|t|)$ in which case we are done. From now on, assume that any two copies of $s$ or $t$ that overlap are aligned. 

Suppose $\eta_i =s$ and $ \mu_j = s$ are aligned where $1\leq i\leq k$ and $1\leq j\leq\ell$. If $\eta_{i+1} = s, 
\mu_{j+1} = t$, then $i+2 \geq k+1$. Indeed, consider three cases:
\begin{itemize}
\item $|s|=|t|$: Then necessarily $i= k$ and $j =\ell$.
\item $|s|<|t|$: If $\eta_{i+2} = s$, then $i+2\geq k+1$ because otherwise $\eta_{i+1}\eta_{i+2} = s^2$ was a subword of $t^*$ (more specifically a subword of $\mu_{j+1}\mu_{j+2} =t^2$). If $\eta_{i+2} = t$ then $\eta_{i+2}$ and $\mu_{j+1}$ are two overlapping not aligned copies of $t$ so $i+2 \geq k+1$.
\item $|s|>|t|$: If $\eta_{i+2} = s$, then $i+2\geq k+1$ because otherwise $\mu_{j+1}\mu_{j+2} = t^2$ was a subword of $s^*$. If $\eta_{i+2} = t$, then $\eta_{i+2}$ and $\mu_{j+2} $ are two overlapping not aligned copies of $t$ so $i+2 \geq k+1$ ($\mu_{j+2}$ overlaps with $\eta_{i+2}$ because otherwise $\mu_{j+1}\mu_{j+2}= t^2$ was a subword of $s^*$).
\end{itemize}

Similarly, if instead $\eta_{i-1} = s, \mu_{j-1} = t$, then $i-2\leq 0$. Similarly we can switch $s$ and $t$.
We are looking for an upper bound of $|w|$. Suppose $w$ contains whole syllable $s^{\alpha_n}$ as a subword for some $n$. There exists $0\leq i\leq k-\alpha_n$ such that $\eta_{i+1} =\dots = \eta_{i+\alpha_n} = s$ for $0\leq i\leq k-\alpha_n$ and $\eta_{i} = t$ and $\eta_{i+\alpha_n+1} = t$. Since $\alpha_n\geq 2 \max\{|s|,|t|\} +1$ there must be a syllable $\mu_j$ contained in the subword spelled by $\eta_{i+1}\cdots\eta_{i+\alpha_n}$, and since $t$ and $s$ are not virtually conjugate $\mu_j = s$. By the previous consideration $\mu_j$ and $\eta_{i'}$ are aligned for some $i+1\leq i'\leq i+\alpha_n$. Since $\alpha_1,\beta_1,\dots,\alpha_{2p},\beta_{2p}$ are all different, we can find $i_0, j_0$ where $i_0\in\{i+1,\dots, i+\alpha_n\}$ such that $\eta_{i_0}=s$ and $\mu_{j_0}=s$ are aligned, and either $\eta_{i_0+1},\mu_{j_0+1}$ or $\eta_{i_0-1},\mu_{j_0-1}$ are different syllables (i.e.\ one of them is $s$ and the other is $t$). By the consideration above, we get that either $i_0+2\geq k+1$, or $i_0-2\leq 0$, which implies that the beginning or the end of the subword $\eta_{i+1}\cdots \eta_{i+\alpha_n} = s^{\alpha_n}$ is less than two syllables away from to the beginning or from the end of $\eta_0\cdots\eta_{k+1}$, respectively. The same happens with a syllable $t^{\beta_n}$ contained in $w$. We conclude that $w$ is always contained in a word of the form $t^2s^{\alpha}t^\beta s^2$ or $s^2t^\beta s^{\alpha}t^2$ for some $\alpha\in\{\alpha_i\}_i$ and $\beta\in\{\beta_j\}_j$, and in particular $|w| \leq \left(\max\{\alpha_i\}+2\right)|s| + \left(\max\{\beta_i\}+2\right)|t|$.

\end{proof}

\begin{proof}[Proof of Proposition~\ref{prop:small cancellation presentation}] First by Lemma~\ref{lem:less than half cancellation} we can assume that for $i=1,\dots,m$ the elements $u_i, v_i$ are non-cancellable. 
Replace the pair $(u_1,v_1)$ and $(u_2, v_2)$  by $(s_1, t_1)$ and $(s_2, t_2)$ respectively as in Lemma~\ref{lem:small cancellation pair}, and continue replacing for each pair of indices $i<j\leq m$. Note that the small cancellation properties of the replaced pairs from Lemma~\ref{lem:small cancellation pair} are preserved when we replace a pair of elements by a pair of its positive subwords. After ${m\choose 2}$ steps we have a collection $\{(s_i, t_i)\}_{i=1}^m$ where for every piece $w$ between a word in $(s_i, t_i)$ and a word in $(s_j, t_j)$ where $i\neq j$ we have $|w|<\max\{|s_i|, |t_i|\}$ and where for any $i$ the elements $s_i$ and $t_i$ are not virtually conjugate. 

Let $r_i(s_i,t_i)=s_i^{\alpha_1^i}t_i^{\beta_1^i}\cdots s_i^{\alpha_{2p}^i}t_i^{\beta_{2p}^i}$ where $\alpha_1^i,\beta_1^i, \dots, \alpha_{2p}^i, \beta_{2p}^i$ are all distinct. Then for each piece $w$ between $r_i$ and $r_j$ where $i\neq j$ we clearly have $|w|<\max\{|s_i|, |t_i|\} <\frac1p |r_i|$. Moreover, if $\min\{\alpha_1^i,\beta_1^i, \dots, \alpha_{2p}^i, \beta_{2p}^i\}>\frac 1 2\max\{\alpha_1^i,\beta_1^i, \dots, \alpha_{2p}^i, \beta_{2p}^i\} +1$ then also for any piece $w$ that lies in $r_i$ in two different ways we also have $|w|<\frac 1 p |r_i|$. Indeed, by Lemma~\ref{lem:small cancellation pair} $r_i$ has the reduced form $gr_i(\bar s_i,\bar t_i)g^{-1}$ where $g\bar s_ig^{-1}, g\bar t_ig^{-1}$ are reduced spellings of $s_i, t_i$ respectively with $\bar s_i,\bar t_i$ cyclically reduced. Let $\bar s_i',\bar t_i'$ be the words that are not proper powers such that $\bar s_i = (\bar s_i')^{n_{s_i}}$ and $\bar t_i = (\bar t_i') ^{n_{t_i}}$, i.e.\ neither $\bar s_i'$ or $\bar t_i'$ is equal to any of its nontrivial cyclic permutations. Also, by Lemma~\ref{lem:small cancellation pair} $(\bar s_i')^{\pm}, (\bar t_i')^{\pm}$ are not conjugate. 

Suppose the piece $w$ is disjoint from $g, g^{-1}$. Then $w$ is a subword of a word in $\bar s_i',\bar t_i'$ and by Lemma~\ref{lem:pieces in single relator} $$\displaystyle |w| \leq \left(\max_j\{n_{s_i}\alpha^i_j\}+2\right)|\bar s_i'| + \left(\max_j\{n_{t_i}\beta^i_j\}+2\right)|\bar t_i'|.$$ It follows that
 
 \begin{align*}
 |w|&\leq (\max_j\{\alpha_j^i,\beta_j^i\}+2)(|\bar s_i|+|\bar t_i|) <2\min_j\{\alpha_j^i,\beta_j^i\}(|\bar s_i|+|\bar t_i|) \\&= \frac 1 {p} \left( 2p\min_j\{\alpha_j^i,\beta_j^i\}(|\bar s_i|+|\bar t_i|) \right)<\frac 1 {p} |r_i|.
 \end{align*}
Finally if $w$ overlaps with the prefix $g$ or suffix $g^{-1}$ then $w$ is a subword of $g\bar s_i^{\alpha_1^i}\bar t_i^{\beta_1^i}$ or $\bar s_i^{\alpha_{2p}^i}\bar t_i^{\beta_{2p}^i}g^{-1}$. If we choose $\alpha_1^i,\beta_1^i, \dots, \alpha_{2p}^i, \beta_{2p}^i$ sufficiently large so $\min_j\{\alpha_j^i,\beta_j^i\}>\frac 1 2\left(\max_j\{\alpha_j^i,\beta_j^i\} +|g|+2\right)$ then we have
 \begin{align*}
|w| \leq |g| + \max_j\{\alpha_j^i,\beta_j^i\}(|\bar s_i|+|\bar t_i|) <2\min_j\{\alpha_j^i, \beta_j^i\} (|\bar s_i|+|\bar t_i|)<\frac 1 p |r_i|.
 \end{align*}
\end{proof}

\section{Proof of Theorem~\ref{thm:main}}\label{sec:proof}
\begin{remark}\label{rem:kar sageev} 
The case $n=2$ of Theorem~\ref{thm:main} can be deduced from the work of Kar and Sageev who study uniform exponential growth of groups acting freely on CAT(0) square complexes \cite{KarSageev16}. They prove that for any two elements $x,y$ there exists a pair of words of length at most $10$ in $x,y$ that freely generates a free semigroup, unless $\langle x,y \rangle$ is virtually abelian. One can construct a small cancellation presentation 
by applying Proposition~\ref{prop:small cancellation presentation} to $\mathcal U = \{(u,v)\mid |u|,|v|\leq 10 \text { and }u,v \text{ are not powers of the same element}\}$. The resulting group cannot act properly on a CAT(0) square complex, since for each pair $u,v$ there is a relator which is a positive word in $u,v$. 
\end{remark}

Let $R_n(x,y)$ be the union of the following pairs for all $k<n$ and $\ell<\ell'\leq K_3$ (where $K_3$ is the constant defined in Lemma~\ref{lem:ramsey}):
\begin{align*}\begin{split}
(x^{n!}, y^{kn!}x^{n!}),\\
(x^{n!}, y^{-kn!}x^{n!}),\\
(x^{-n!}, y^{kn!}x^{-n!}),\\
(x^{-n!}, y^{-kn!}x^{-n!}),\\
\end{split}
\begin{split}
(x^{-n!}y^{-k}x^{n!}y^{k\ell}, x^{-n!}y^{-k}x^{n!}y^{k\ell'}),\\
(y^{-k}x^{-n!}y^{k\ell}x^{n!}, y^{-k}x^{-n!}y^{k\ell'}x^{n!}),\\
(x^{-k}y^{k\ell}, x^{-k}y^{k\ell'}),\\
(x^{-n!}, y^{k\ell} x^{-n!}),\\
(x^{n!}, y^{k\ell}x^{n!}).
\end{split}
\end{align*}
Let $\mathcal R_1(x,y) = R_1(x,y)\cup R_1(y,x)$. Let 
\[
\mathcal R_n(x,y) = R_n(x,y)\cup R_n(y,x)\cup \mathcal R_{n-1}(y^N, x^{-n!}y^Nx^{n!}) \cup \mathcal R_{n-1}(x^N, y^{-n!}x^Ny^{n!})
\]
where $N=n!K_3!$.

\begin{lem}[The Main Lemma]\label{lem:main}
Suppose $\langle x,y\rangle $ acts freely on an $n$-dimensional CAT(0) cube complex. Then one of the following holds:
\begin{itemize}
\item one of the pairs in $R_n(x,y)$ freely generates a free semigroup, or
\item either $y^N$ and $x^{-n!}y^Nx^{n!}$, or $x^N$ and $y^{-n!}x^Ny^{n!}$ stabilize a hyperplane, or
\item the group $\langle x^N,y^N\rangle$ is virtually abelian.
\end{itemize}
\end{lem}

\begin{proof}
 Without loss of generality we may assume that the action of $\langle x,y \rangle$ is without hyperplane inversions, as we can always subdivide $X$ to have this property of the action. Let $\gamma_x,\gamma_y$ be axes of $x,y$ respectively.

Suppose there exists a hyperplane $\h\in\sk(x)-\sk(y)$. By Lemma~\ref{lem:all or nothing}, $y$ does not skewer $x^{in!}\h$ for any $i\in \Z$ unless one of the pairs in $R_n(x,y)$ freely generates a free semigroup. Without loss of generality (by possibly renaming some $x^{in!}\h$ as $\h$) we can assume that $\gamma_y\subset h\cap x^{n!}h^*$.

If $y^{N}\h =\h $ and $y^{N}x^{n!}\h =x^{n!}\h $ then the subgroup $\langle y^{N}, x^{-n!}y^{N}x^{n!}\rangle$ preserves $\h$. We are now assuming that this is not the case, i.e.\  at least one of $\h$ and $x^{n!}\h$ is not preserved by $y^{N}$.

Suppose that $y^N$ does not stabilize $\h$.  Let $k\leq n$ be minimal such that $y^{k}x^{n!}\h$ and $x^{n!}\h$ are disjoint or equal and let $\ell<\ell'\leq K_3$ such that $\{\h, y^{k\ell}\h, y^{k\ell'}\h\}$ are pairwise disjoint (no two can be equal since $y^N$ does not stabilize $\h$). 
If $y^{k}x^{n!}\h\neq  x^{n!}\h$, then we have $y^{k}x^{n!}h\subset x^{n!}h^*$, and thus also $x^{-n!}y^{k}x^{n!}h\subset h^*$. Since $y^{k\ell}h^*\subset h$ and $y^{k\ell'}h^*\subset h$ there is a ping-pong triple $\{x^{-n!}y^{k}x^{n!}h^*, y^{k\ell}h^*, y^{k\ell'}h^*\}$. See Figure~\ref{fig:case1}.
\begin{figure}
\begin{tikzpicture}
\draw[line width=0.5mm] (0,-2) to (0,2);
\draw[->, line width=0.5mm] (0,0) -- (0.5,0) node [midway, label=below:$h$] {};
\draw[->] (1,-1.8) -- (1, 1.8) node [pos=0.9, label={[label distance=-5pt]left:$\gamma_y$}] {};
\draw[line width=0.5mm] (-2,-2) to (-2,2);
\draw[->, line width=0.5mm] (-2,0) -- (-2.5,0) node [midway, label=left:$x^{-n!}y^{k}x^{n!}h$] {};
\draw[line width=0.5mm] (2.5, 1) to[in=270, out=180] (1.5, 2);
\draw[->, line width=0.5mm] (1.8,1.3) -- (1.5, 1) node [midway, label=below:$y^{k\ell}h$] {}; 
\draw[line width=0.5mm] (2.5, -1) to[in=90, out=180] (1.5, -2);
\draw[->, line width=0.5mm] (1.8,-1.3) -- (1.5, -1) node [midway, label=above:$y^{k\ell'}h $] {};
\end{tikzpicture}
\caption{The case where $y^{N}\h \neq\h $ and $y^{k}x^{n!}\h \neq x^{n!}\h $.}\label{fig:case1}
\end{figure}
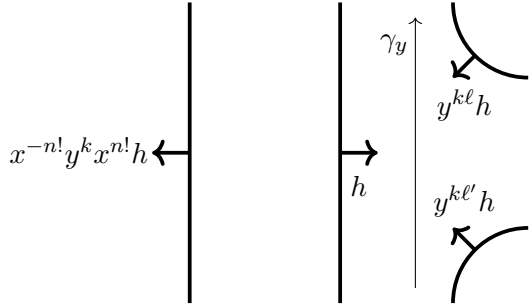
Now suppose $y^{k}x^{n!}\h = x^{n!}\h$.
We have
$y^{k\ell}h^*\subset x^{n!}h^*$ because $h^*\subset x^{n!}h^*$, 
and thus $\{x^{n!}h^*, h^*, y^{k\ell}h^*\}$ is a ping-pong triple.
Analogously, if $y^N$ does not stabilize $x^{n!}\h$ 
then one of $\{x^{n!}y^kh,y^{k\ell}x^{n!}h,y^{k\ell'}x^{n!}h\}$ 
and $\{h, x^{n!}h, y^{k\ell}x^{n!}h\}$ is a ping-pong triple for some $k\leq n$ and $\ell<\ell'\leq K_3$.
\begin{com}$h$ or $h^*$?\end{com}

Similarly, if there exists a hyperplane $\h\in\sk(y)-\sk(x)$, 
then one of the pairs in $\mathcal R_n(x,y)$ freely generates a free semigroup 
or $\langle x^N, y^{-n!}x^Ny^{n!}\rangle$ stabilizes a hyperplane. 
Otherwise $\sk(x) = \sk(y)$, which we now assume is the case.

Suppose there exists a hyperplane $\h$ separating $\gamma_x,\gamma_y$ 
that is not stabilized by either $x^{K_3!}$ or $y^{K_3!}$. 
Let $k\leq n$ be minimal such that $x^kh\subset h^*$ for appropriate choice of halfspace $h$ of $\h$. 
Let $\ell, \ell'\leq K_3$ such that $\{\h, y^{k\ell}\h, y^{k\ell'}\h\}$ are pairwise disjoint. 
The triple $\{x^kh^*, y^{k\ell}h^*, y^{k\ell'}h^*\}$ is a ping-pong triple. 

We can now assume that every hyperplane separating any two axes of $x$ and $y$
is stabilized by $x^{K_3!}$ or $y^{K_3!}$. If a hyperplane $\h$ is stabilized by $x^{K_3!}$ then there are axes of $x^{K_3!}$ in both halfspaces  $h, h^*$. In particular, no hyperplane separates $\Min^0(x^{K_3!})$ and $\Min^0(y^{K_3!})$, hence $\Hull(\Min^0(x^{K_3!}))\cap\Hull(\Min^0(y^{K_3!}))\neq \emptyset$. 
Let $p$ be a $0$-cube in the intersection $\Hull(\Min^0(x^{K_3!}))\cap\Hull(\Min^0(y^{K_3!}))$. 
By Lemma~\ref{lem:convex hull}, $p$ lies on axes of both $x^{N}$ and $y^{N}$. The subcomplexes $\Hull(\langle x^N\rangle p)$ and $\Hull(\langle y^N\rangle p)$ both contain $p$ and are dual to the same set of hyperplanes $\sk(x) = \sk(y)$, and therefore are equal. 
That subcomplex is invariant under the action of $x^N$ and $y^N$ and so it is a minimal $\langle x^{N},y^{N}\rangle$-invariant convex subcomplex, and $\langle x^{N},y^{N}\rangle$ acts freely on it.
By Lemma~\ref{lem:convex hull} $\Hull(\gamma)$ embeds in $\E^k$ and by Lemma~\ref{lem:euclidean} the group $\langle x^{N},y^{N}\rangle$ is virtually abelian.
\end{proof}

In the following proof $|w|_*$ denotes the minimal number of syllables of the form $x^{\pm*}, y^{\pm*}$ in a spelling of $w$.
\begin{proof}[Proof of Theorem~\ref{thm:main}]
Let $G$ be a group given by the $C'(1/p')$ presentation from Proposition~\ref{prop:small cancellation presentation} with $\mathcal U = \mathcal R_n(x,y)$ where $p' = \max\{p, 8\cdot 3^n\}$. In particular, $G$ is an infinite, torsion-free, non-elementary hyperbolic group \cite[Thm 4.4]{LS77},\cite{Gromov87}. In particular, any nontrivial virtually abelian subgroup is isomorphic to $\Z$. Since $p'\geq p$ the group $G$ is $C'(1/p)$. Suppose that $G$ acts properly (and hence freely) on an $n$-dimensional CAT(0) cube complex.

By definition of $G$ none of the pairs in $\mathcal R_n(x,y)$ can freely generate a free semigroup since there is a relator in the presentation of $G$ associated to each pair. Since the presentation of $G$ is $C'(1/6)$ it can be concluded from the Greendlinger's Lemma \cite[Thm 4.4]{LS77} that the subgroup $\langle x^N, y^N\rangle$ is not isomorphic to $\Z$  and hence not virtually abelian, so by Lemma~\ref{lem:main} one of the pairs $y^N, x^{-n!}y^Nx^{n!}$ or $x^N, y^{-n!}x^Ny^{n!}$ stabilizes a hyperplane and thus these two elements act on an $(n-1)$-dimensional CAT(0) cube complex. Since $R_{n-1}(y^N,x^{-n!}y^Nx^{n!})\subset  \mathcal R_n(x,y)$ and $R_{n-1}(x^N,y^{-n!}x^Ny^{n!})\subset\mathcal R_n(x,y)$ we can apply Lemma~\ref{lem:main} again and we conclude that either one of $\langle y^N, x^{-n!}y^Nx^{n!}\rangle$ and $\langle x^N, y^{-n!}x^Ny^{n!}\rangle$ is a $\Z$ subgroup, or an appropriate pair of elements stabilizes a hyperplane. We can keep applying Lemma~\ref{lem:main}. As long as the pair of elements $u,v$ stabilizes a hyperplanes, then by Lemma~\ref{lem:main} one of the pairs $v^{N}, u^{-n!}v^{N}u^{n!}$ or $u^{N}, v^{-n!}u^{N}v^{n!}$ generates a $\Z$ subgroup or stabilizes a hyperplane. Since the elements that Lemma~\ref{lem:main} provides are both conjugates of $x$ or $y$, we conclude that $u$ and $v$ at each step are some conjugates of one of the original generators $x$ and $y$. In particular, $|u^k|_* = |u|_*$ and $|v^{k}|_* = |v_*|$ for any $k>0$. Also, \begin{align*}|v^{-n!}u^{N}v^{n!}|_*&\leq |v^{-n!}|_*+|u^{N}|_*+|v^{n!}|_*\\
&= |v|_*+|u|_*+|v|_*\leq 3\max\{|u|_*,|v|_*\},
\end{align*}
and similarly $|u^{-n!}v^{N}u^{n!}|_*\leq3\max\{|u|_*,|v|_*\}$.
By applying Lemma~\ref{lem:main} up to $n$ times, we eventually get a pair of elements $u_0,v_0$ that generates a $\Z$ subgroup and we have $|u_0|_*,|v_0|_*\leq 3^n$. 
In particular, $u_0^k = v_0^{k'}$ for some $k,k'\neq 0$ and we have $|u_0^kv_0^{-k'}|_*\leq 2\cdot 3^n$.
By Greendlinger's Lemma~\cite{LS77} some subword $w$ of $u_0^kv_0^{-k'}$ must be also a subword of some relator $r$ with $|w|\geq\frac 1 2|r|$. On one hand $|w|_*\leq 2\cdot 3^n$. On the other hand, the length of each syllable of the form $x^{\pm*}$ or $y^{\pm*}$ in $r$ is at most $1+\frac 1 {p'}|r|<\frac2{p'}|r|$ because if $x^k$ is a subword of $r$ then $x^{k-1}$ is a piece in $r$ and the same for $y$. Thus for any subword $w'$ of $r$ of length at most $\frac {|r|}2$ we have $|w'|_*>\frac {p'} 4$. Since $\frac {p'}4\geq 2\cdot 3^n$ we get a contradiction.
\end{proof}

\bibliographystyle{alpha}
\bibliography{../../../../kasia}
\end{document}